\documentclass{amsart}
\usepackage{amssymb}
\newtheorem{theorem}{Theorem}

\newtheorem{corollary}{Corollary}

\newtheorem{definition}{Definition}
\newtheorem{example}{Example}

\newtheorem{lemma}{Lemma}

\newtheorem{proposition}{Proposition}
\newtheorem{remark}{Remark}

\makeindex

\begin{document}
\title{A note on primary spectrum over commutative rings}
\author{NESL\.{I}HAN AY\c{S}EN {\"{O}}ZK\.{I}R\.{I}\c{S}C\.{I}}
\thanks{Corresponding Author: N. AY\c{S}EN {\"{O}}ZK\.{I}R\.{I}\c{S}C\.{I}}
\address{Yildiz Technical University, Department of Mathematics, Esenler, \.{I}stanbul, TURKEY}
\email{aozk@yildiz.edu.tr}
\author{ZEL\.{I}HA K{\i}L{\i}\c{C}}
\address{The University of North Carolina at Chapel Hill, Chapel Hill, NC, USA}
\email{zelihakilicc@gmail.com}
\author{SUAT KO\c{C}}
\address{Marmara University, Department of Mathematics, Kad{\i}k{\"{o}}y, \.{I}stanbul, TURKEY}
\email{suat.k.oc@hotmail.com}

\maketitle

\begin{abstract}
In this work we define a primary spectrum of a commutative ring $R$ with its Zariski topology $\mathfrak{T}$. We introduce several properties and examine some topological features of this concept. We also investigate differences between the prime spectrum and our primary spectrum.
\end{abstract}

\textbf{2000}\textbf{\textit{ Mathematics Subject Classification.}} {13E99, 13A99, 13A15.}


\bigskip

\section{Introduction}\label{intro}
Let $R$ be a commutative ring with a nonzero identity. $Max(R)$, $Spec(R)$ and $Prim(R)$ denote the set of maximal, prime and primary ideals of $R$, respectively. The Zariski topology on $Spec(R)$  is defined to be the topology whose closed sets are the subsets of the form $V(I)$, denoting the set of prime idelas of $R$ containing $I$. The family $\{V(I):I$ is an ideal of $R\}$ satisfies the axioms of closed sets of a topology $\mathfrak{T}$ on $Spec(R)$. The topological space $(Spec(R), \mathfrak{T})$ is called the prime spectrum of $R.$ \\
The notion of primary spectrum on principal ideal domains was examined as a generalization of prime spectrum in [1]. The variety, i.e. the operation $V=V_{Q}$ on the subsets of $R$, considered in [1] is $ V_{Q}(E)=\{I: I$ is a primary ideal of $R$ such that $E\subset I\}$. They showed that the set of primary ideals of a principal ideal domain $R$ can be endowed with a topology called primary spectrum of $R$. \\
In this study, we define a different variety in a ring $R$ which is not necessarily a principal ideal domain. We denote our variety by $V_{rad}(I):=\{Q\in Prim(R): I\subseteq\sqrt{Q}$ for any ideal $I$ of $R\}$. The family $\{V_{rad}(I): I$ is an ideal of $R\}$ satisfies the axioms of closed sets for a topology $\mathfrak{T}$ on $Prim(R)$ called the Zariski topology on $Prim(R)$, and the space $(Prim(R), \mathfrak{T})$ is the primary spectrum of $R$ denoted by $Prim(R).$\\
In section 2, we construct the primary spectrum of a ring $R$ and compare it with its prime spectrum. Additionally, we give some examples to clarify the difference between the prime and primary spectrum. In section 3, we investigate some topological properties of $Prim(R)$ and provide various applications of those properties. \\

\bigskip

\bigskip

\section{On The Primary Spectrum Of Commutative Rings}
\definition{$Q$ is said to be a primary ideal of $R$ if $rs\in Q$ but $r\not\in Q$ implies that there exists $n\in N$ such that $s{^{n}}\in Q$ for $r,s\in R$ where $Q$ is a proper ideal of $R$.}

\definition{Let $Prim(R)$ be the  set of all primary ideal of $R.$ We define primary variety for any subset $S$ of $R$ as
\begin{center}
$V_{rad}(S)= \{Q\in Prim(R):S\subseteq\sqrt{Q}\}.$
\end{center}
}

\remark{Let $I, J$ be the ideals of $R$ and $a\in R$.
\begin{enumerate}
\item {It is clear that if $I=(S),$ then $V_{rad}(S)=V_{rad}(I).$ If $S=\{a\},$ we write $V_{rad}(a)=V_{rad}(\{a\})$ and we have $V_{rad}(a)=V_{rad}(Ra).$}
\item {If $I\subseteq J,$ then $V_{rad}(J)\subseteq V_{rad}(I).$ }
\item {$V_{rad}(I)= V_{rad}(\sqrt{I}).$}
\end{enumerate}
}
\proposition{Let $\{S_{i}: i\in \Lambda\}$ be a family of subsets of $R$ and $I, J, I_{i}$ s be ideals of $R$ where $i\in \Lambda$ for any index set $\Lambda$. Then the followings hold:
\begin{enumerate}
  \item {$V_{rad}(0_{R})=Prim(R)$ and $V_{rad}(1_{R})={\O}$ }
  \item {$V_{rad}(I\cap J)=V_{rad}(IJ)=V_{rad}(I)\cup V_{rad}(J).$ }
  \item {$V_{rad}(\bigcup_{i\in \Lambda}{S_{i}})=\bigcap_{i\in \Lambda}V_{rad}(S_{i}).$ In particular, $V_{rad}(\sum_{i\in \Lambda}{I_{i}})=\bigcap_{i\in \Lambda}V_{rad}(I_{i}).$ }
\end{enumerate}
}

\corollary{The collection $\{V_{rad}(I):I$ is an ideal of $R\}$ satisfies the axioms of closed sets of a topology $\mathfrak{T}$ on $Prim(R)$. This topology is said to be Zariski topology on $Prim(R).$ We call the topological space $(Prim(R),\mathfrak{T})$ as primary spectrum of $R.$}\\

Various examples have been provided to emphasize some differences between the prime and primary spectrum of $R$. In Example 1, we observe that the set of $V(I)$ can be finite as opposed to $V_{rad}(I)$ that is infinite for any proper ideal $I$ of a ring $R.$
\bigskip

\example{Let $R=Z.$ It is known that $Z$ is a principal ideal domain and all primary ideals of $Z$ is of the form of $(0)$ and $(p^{k}),$ where $k\in N$ and $p$ is a prime number. Consider the prime factorization of $n\in Z$ and let $n=p_{1}^{t_{1}}p_{2}^{t_{2}}...p_{m}^{t_{m}}$ for any integer $m$ and prime integers $p_{i},$ $t_{i}\in N$, $i\in \{1,2,...,m\}.$ Thus $V_{rad}(n)=\{(p_{i}^{k_{i}}): k_{i}\in N$, $i\in \{1,2,...,m\}\}$; however, $V(n)=\{(p_{i}):i\in \{1,2,...,m\}\}.$\\

Now we provide some examples of chain rings for further analyzes of our variety.
\example{Let $R=Z_{p^{m}}.$ Every ideal in $Z_{p^{m}}$ is of the form of $(p^{i}),$ where $i\le m.$ It can be readily noted that $Spec(Z_{p^{m}})=\{(p)\}$ and every ideal in $Z_{p^{m}}$ is primary, that is, $Prim(Z_{p^{m}})=\{(0)\}\cup \{(p^{i}):i<m, i\in N\}.$
Therefore, $V_{rad}(p^{i})=Prim(Z_{p^{m}})$ for any proper ideal $(p^{i})$ of $Z_{p^{m}}.$ However, $V(I)=Spec(Z_{p^{m}})$ for any proper ideal $I$ in $Z_{p^{m}}.$
\example{Let consider the Galois ring $R=Z_{p^{s}}[X]/(h(X)),$ where $h(X)$ be a monic basic irreducible polynomial of degree $m$ in $Z_{p^{s}}[X].$ It is known that chain ring $R$ can be denoted by $GR(p^{s},p^{sm}).$ (For more details, see [2].) The set of ideals of $R$ consists of principal ideals $(0),(p),(p^{2}),..., (p^{s-1}),(1_{R})$. This tells us that every proper ideal of $R$ is primary, because radical of all ideals is maximal. Hence $V_{rad}(I)=Prim(GR(p^{s},p^{sm}))$ for any proper ideal $I$ of $GR(p^{s},p^{sm})$} and it is obvious that $V(I)=\{(p)\}=Spec(GR(p^{s},p^{sm})).$\\

\textmd{Since all finite chain rings are isomorphic to the finite chain rings in the second or third example, we completed to examine the primary variety of all finite chain rings by means of those examples.}

\example{Let $k$ be a field and define $R=k[X]/(X^{n})$. It is known that $R$ is an Artinian local ring with dimension $0.$ Every ideal in $R$ is of the form $(X^{i})/(X^{n}),$ where $i\le n$ and $(X)/(X^{n})$ is the unique prime ideal of $R.$ Since $R$ is not trivial and has the property that every zero divisor in $R$ is nilpotent, $Prim(R)=\{(0)\} \cup \{(X^{i})/(X^{n}):i< n\}$ that is, every ideal of $R$ is primary. So we obtain $V_{rad}(I)=Prim(R).$ As the set $Spec(k[X])$ consists of irreducible polynomials in $k[X]$, $Spec(k[X]/(X^{n}))=\{(X)/(X^{n}):n\in N\}$. Hence $V(I)=Spec(R)$ for any proper ideal $I$ of $R.$
\section{Some properties of Primary Spectrum}

Recall that any open subset of $Prim(R)$ is of the form $Prim(R)\setminus V_{rad}(S)$ for any subset $S$ of $R.$}\\

In the following theorem we give a basis for $Prim(R).$

\theorem{Let $R$ be a ring and $X_{r}=Prim(R)\setminus V_{rad}(r).$ The family $\{X_{r}:r\in R\}$ forms a base of the Zariski topology on $Prim(R).$}
\proof{Assume that $U$ is any open set in $Prim(R).$ Then for a subset $S$ of $R$, we have $U=Prim(R)\setminus V_{rad}(S)=Prim(R)\setminus V_{rad}(\bigcup_{r\in S}\{r\})=Prim(R)\setminus \bigcap_{r\in S}V_{rad}(r)=\bigcup_{r\in S}(Prim(R)\setminus V_{rad}(r))=\bigcup_{r\in S}X_{r}.$ Thus,  $\{X_{r}\}_{r\in R}$ is a basis for the Zariski topology.\\

Note that, it follows from the definition that, $X_{0}={\O},$ $X_{1}=Prim(R)$ and more generally $X_{r}=Prim(R)$ for every unit $r\in R.$}
\theorem{Let $R$ be a ring. Then the followings hold for any $r,s\in R$ and the open sets $X_{r}$ and $X_{s}:$}
\begin{enumerate}
  \item {$(\sqrt{rR})=(\sqrt{sR})$ if and only if $X_{r}=X_{s}$,}
  \item {$X_{rs}=X_{r}\cap X_{s}$,}
  \item {$X_{r}=\O \Leftrightarrow r$ is nilpotent,}
  \item {$X_{r}$ is quasi compact.}
\end{enumerate}
\proof{
(1),(2) and (3) are obvious.
(4) Let $r,s\in R$ and $\lambda$ be an index set. Assume that $\{X_{s_{i}}:i\in\lambda\}$ is an open cover of $X_{r}.$ Therefore, $X_{r}\subseteq\bigcup_{i\in\lambda}X_{s_{i}}=\bigcup_{i\in\lambda}(Prim(R)\setminus V_{rad}(s_{i}))=Prim(R)\setminus\bigcap_{i\in\lambda}V_{rad}({s_{i}})=Prim(R)\setminus V_{rad}(\bigcup_{i\in\lambda} {s_{i}}),$ that is, $V_{rad}(\bigcup_{i\in\lambda} s_{i})\subseteq V_{rad}(r)=V_{rad}(\sqrt{rR}).$ Hence, $\sqrt{rR}\subseteq\sqrt{(\bigcup_{i\in\lambda}\{s_{i}\})}.$ Then $r^{n}\in(\bigcup_{i\in\lambda}\{s_{i}\})$ for some $n\in N.$ There exists a finite subset $\Delta\subseteq\lambda$ such that $r^{n}=\sum_{j\in\Delta}t_{j}s_{j},$ for any $t_{j}\in R$ and $j\in\Delta.$ Thus, $(rR)^{n}\subseteq(\{s_{j}:j\in\Delta\}),$ that is, $V_{rad}(\{s_{j}:j\in\Delta\})\subseteq V_{rad}(r^{n})=V_{rad}(r)$. Hence $V_{rad}(\sum_{j\in\Delta}(s_{j}))=\bigcap_{j\in\Delta}V_{rad}(s_{j})\subseteq V_{rad}(r).$ So, $Prim(R)-V_{rad}(r)\subseteq Prim(R)-\bigcap\limits_{j\in \Delta }V_{rad}\left( s_{j}\right) =\bigcup_{j\in\Delta} \left(Prim\left( R\right)-V_{rad}\left( s_{j}\right) \right)=\bigcup\limits_{j\in \Delta }X_{s_{j}}$. As a consequence, $X_{r}\subseteq \bigcup_{j\in\Delta}X_{s_{j}}.$
}
\theorem{Let $R$ be a ring. Then $Prim(R)$ is quasi-compact.}
\proof{It can be seen directly from Theorem 2(4).}\\

The following condition is given by Hwang and Chang for the prime spectrum in [3]:\\

(*) If $X_{r}\subseteq{\displaystyle\bigcup\limits_{\alpha\in\Lambda}} X_{s_{\alpha}}$ where $r,s_{\alpha}\in R$ $(\alpha\in\Lambda)$ are nonzero elements of $R$, then $X_{r}\subseteq X_{s_{\alpha}}$ for some $\alpha\in\Lambda.$

We use this property by taking $X_{r}=Prim(R)\setminus V_{rad}(r)$ for primary spectrum as below.

\begin{theorem}
Let every nonzero prime ideal of $R$ is maximal. Then $R$ satisfies (*) if and only if $R$ has at most two nonzero prime ideals.
\end{theorem}

\begin{proof}
Let $R$ be a ring satisfying the (*) condition. Assume that $M_{1},M_{2}$ and
$M_{3}$ are three distinct maximal ideals of $R$. Then $a+b=1$ for some $a\in
M_{1}$ and $b\in M_{2}$. Hence $r=ra+rb$ for any $r\in M_{3}-\left(  M_{1}\cup
M_{2}\right)  $. It follows that $X_{r}\subseteq
X_{ra}\cup X_{rb}$. But $X_{r}\nsubseteq X_{ra}$ and $X_{r}\nsubseteq X_{rb}$ since $M_{1}\in X_{r}-X_{ra}$ and $M_{2}\in X_{r}-X_{rb}$. So we get a contradiction.\\
For the necessary condition, let $R$ has at most two nonzero prime ideals, say
$P_{1}$ and $P_{2}$, and assume that $X_{r}\nsubseteq X_{s_{\alpha}}$ for all
$\alpha\in\Lambda$ where $r,s_{\alpha}\in R.$ Let the ideals $Q_{j,i}$
represent $P_{j}$-primary ideals in $R$ for $j=1,2$ and $i \in S,$ where $S$ is an index set. Then there exists a
$P_{j}$-primary ideal contained in $X_{r}$ but not $X_{s_{\alpha}}$ for all
$\alpha\in\Lambda$. Say $Q_{1,k}$ where $k\in S$. Without loss of generality,
take $j=1$. Then $r\notin P_{1}$, that is, every $P_{1}$-primary ideal is in
$X_{r}$. Thus $r\in P_{2}$, if $r$ is nonunit in $R$. Namely,
$Q_{2,i}\notin X_{r}$ for all $i\in S$ and so $X_{r}=\left\{  Q_{1,i}\right\}
_{i\in S}$. Since $s_{\alpha}\in P_{1}$, we obtain $
{\displaystyle\sum\limits_{\alpha\in\Lambda}}
\left(  s_{\alpha}\right)  \subseteq P_{1}$. It follows that $\left\{
Q_{1,i}\right\}  _{i\in S}\subseteq V\left({\displaystyle\sum\limits_{\alpha\in\Lambda}}
\left(  s_{\alpha}\right)  \right)  =
{\displaystyle\bigcap\limits_{\alpha\in\Lambda}}
V\left(  s_{\alpha}\right)  $. Hence, $\left\{  Q_{1,i}\right\}  _{i\in S}
\cap\left(
{\displaystyle\bigcup\limits_{\alpha\in\Lambda}}
X_{s_{\alpha}}\right)  =\varnothing.$ Consequently, we get $X_{r}\nsubseteq
{\displaystyle\bigcup\limits_{\alpha\in\Lambda}}
X_{s_{\alpha}}$. If $r$ is unit, then $X_{r}=Prim(R)$ and so $X_{r}\nsubseteq
{\displaystyle\bigcup\limits_{\alpha\in\Lambda}}
X_{s_{\alpha}}$. For the another case, that is, if $R$ has only one nonzero
prime ideal, then $r$ is unit element in $R$ with above assumptions. It
follows that $X_{r}\nsubseteq%
{\displaystyle\bigcup\limits_{\alpha\in\Lambda}}
X_{s_{\alpha}}$. The last case is that there is no nonzero prime ideal in $R$.
Thus the only maximal ideal of $R$ is zero ideal and then we get $r$ is unit
element since $r\neq0$. In this case we also obtain $s_{\alpha}=0$, that is $X_{s_{\alpha
}}=\varnothing$ for all $\alpha\in\Lambda$. Then we get $X_{r}\nsubseteq
{\displaystyle\bigcup\limits_{\alpha\in\Lambda}}X_{s_{\alpha}}$.
\end{proof}

In the following lemma, Arapovic characterizes the embeddability of a ring
into a zero-dimensional ring with two properties in [4].

\begin{lemma}
A ring $R$ is embeddable in a zero-dimensional ring if
and only if $R$ has a family of primary ideals $\left\{ Q_{\lambda }\right\}
_{\lambda \in \Lambda }$, such that:\newline
A1. $\bigcap\limits_{\lambda \in \Lambda }Q_{\lambda }=0$, and\newline
A2. For each $a\in R$, there is $n\in
\mathbb{N}$ such that for all $\lambda \in \Lambda $, if $a\in \sqrt{Q_{\lambda }}$,
then $a^{n}\in Q_{\lambda }$.
\end{lemma}

The condition (A2) is very useful by the above lemma. After this work, Brewer and Richman [5] give an equivalent condition as follows:

\begin{lemma}
A family $\left\{ I_{\lambda }\right\} _{\lambda \in
\Lambda }$ of ideals in a ring $R$ satisfies (A2) if and only if for each
(countable) subset $\Gamma \subset \Lambda $, $\sqrt{\underset{\lambda \in
\Gamma }{\ \bigcap }I_{\lambda }}=\underset{\lambda \in \Gamma }{\ \bigcap }%
\sqrt{I_{\lambda }}$.
\end{lemma}

Moreover, Brewer and Richman [5] prove the following theorem as a characterization of zero-dimensional rings.

\begin{theorem}
The following conditions are equivalent:
\newline
(i) $R$ is zero-dimensional.\newline
(ii) Condition (A2) holds for the family of all ideals of $R$.\newline
(iii) Condition (A2) holds for the family of all primary ideals of $R$.
\end{theorem}

\definition{ $\xi(Y)$ denotes the intersection of all elements of $Y,$ where $Y$ is any subset of a topological space $X.$}
\theorem{Let $R$ be a zero dimensional ring and $Y\subseteq Prim(R).$ Then $V_{rad}(\xi(Y))=Cl(Y).$ Hence, $Y$ is closed if and only if $V_{rad}(\xi(Y))=Y.$}
\proof{Let $Q\in Y.$ Then $\xi(Y)\subseteq Q\subseteq\sqrt{Q}$ by definition of $\xi(Y)$. Thus, $Q\in V_{rad}(\xi(Y)),$ that is, $Y\subseteq V_{rad}(\xi(Y)).$ Hence we obtain $Cl(Y)\subseteq V_{rad}(\xi(Y)).$ For the reverse inclusion, let $V_{rad}(I)$ be a closed subset of $Prim(R)$ including $Y$. Then we have $I\subseteq \sqrt{Q}$ for all $Q\in Y$. This gives us $I\subseteq \sqrt{\xi \left( Y\right) }
$ by Theorem 5. Let $Q^{\prime } \in V_{rad}(\xi(Y))$. Then $\xi(Y)\subseteq \sqrt{Q^{\prime }}$. Hence we get  $I\subseteq \sqrt{\xi \left( Y\right) }\subseteq \sqrt{Q^{\prime }}$. This implies that $Q^{\prime }\in V_{rad}(I)$, that is, $V_{rad}\left( \xi \left(
Y\right) \right) $ is the smallest closed subset of $Prim(R)$ which
includes $Y$.}
\proposition{Let $I\in Prim(R).$ Then the followings hold:
\begin{enumerate}
\item {$Cl(\{I\})=V_{rad}(I).$}
\item {$J\in Cl(\{I\})$ if and only if $I\subseteq\sqrt{J}$ for any $J\in Prim(R).$}
\end{enumerate}
}
\proof{
\begin{enumerate}
\item{If we take $Y=\{I\},$ we obtain $Cl(\{I\})=V_{rad}(I)$ by Theorem 6. Note that $R$ is not necessarily a zero-dimensional ring because $Y$ is a singleton set.}
\item{It is an immediate consequence of (1).}
\end{enumerate}
}
\corollary{$Prim(R)$ is a $T_{0}-$space if and only if for any two ideals $I$ and $J$ in $Prim(R)$, $V_{rad}(I)=V_{rad}(J)$ implies that $I=J$.}
\remark{It is known that $Spec(R)$ is always a $T_{0}-$space for the Zariski topology, but $Prim(R)$ is not necessarily a $T_{0}-$space for any ring $R$.
\example{Let $R=Z.$ Then $Prim(Z)=\{0\}\cup\{(p^{n}): p$ is prime$\}$ and $V_{rad}(p^{n})=\{Q:(p^{n})\subseteq\sqrt{Q}\}.$ Therefore, for any $n,m\in N$ we have $Cl(\{p^{n}\})=V_{rad}(p^{n})=V_{rad}(p^{m})=Cl(\{p^{m}\}).$ However, $p^{n}\ne p^{m},$ for distinct $n, m.$  }
\example{Let $R=Z\times Z.$ It is easily proved that $Prim(R)=\{(p^{i})\times Z: p$ is prime$, i\in N\}\cup\{Z\times (q^{j}): q$ is prime$, j\in N\}.$ Furthermore, $Cl((p^{i})\times Z)=V_{rad}((p^{i})\times Z)=\{(p^{n})\times Z:(p^{i})\times Z\subseteq\sqrt{(p^{n})\times Z}=(p)\times Z\}.$ Clearly, $Cl((p^{i})\times Z)=Cl((p^{j})\times Z).$ However, $(p^{i})\times Z\ne (p^{j})\times Z,$ for $i\ne j.$}
}
\definition{[6] If every primary ideal is a maximal ideal in a ring $R$, then $R$ is called a $P$-ring.}
\begin{proposition}
$R$ is a $P$-ring if and only if $Prim(R)$ is a $T_{0}-$space.
\end{proposition}

\begin{proof}
The necessary condition is clear. For the sufficient condition let $Prim(R)$ be a $T_{0}-$ space and $ Q \in Prim(R)$. Then we have
$Cl(\{Q\}) =V_{rad}\left( Q\right) =V_{rad}\left( \sqrt{Q}\right)=Cl(\{\sqrt{Q}\}) $. This gives us $Q=\sqrt{Q}\ $by the assumption. Hence we obtain $Q\ $is a prime ideal and $R\ $is a $P$-ring.
\end{proof}

\proposition{$Prim(R)$ is a $T_{2}-$space if and only if $R$ is a $P$-ring.}
\proof{Let $Prim(R)$ be a $T_{2}$-space. Then $Prim(R)$ is a $T_{0}$-space and so we have $R$ is a $P$-ring by Proposition 3. Conversely, let $R$ be a $P$-ring. Then $Spec(R)$ coincides with $Prim(R).$ Since $R$ is zero dimensional, $Spec(R)$ is a $T_{2}$-space and so is $Prim(R)$.}\\

We get the corollary given below by combining last two propositions:

\corollary{The following statements are equivalent:

(i)\ $R$ is a $P$-ring.

(ii)\ $Prim(R)$ is a $T_{2}$-space.

(iii)\ $Prim(R)$ is a $T_{1}$-space.

(iv)\ $Prim(R)$ is a $T_{0}$-space.

\definition{A topological space $X$ is irreducible if $X$ is nonempty and $X$ can not be expressed as a union of two proper closed subsets of $X.$ Equivalently, $X$ is irreducible if $X$ is nonempty and any two nonempty open subsets of $X$ intersect.}
\proposition{$V_{rad}(I)$ is an irreducible closed subset of $Prim(R)$ for $I\in Prim(R).$}
\proof{It is an immediate consequence of Proposition 2(1).}
\begin{theorem}
$Prim(R)$ is irreducible if and only if nilradical of $R$, $N(R)$, is a primary ideal of $R.$
\end{theorem}

\begin{proof}
Suppose that $N(R)$ is a primary ideal of $R.$ Let $U$ and $V$ be nonempty open subsets of $Prim(R)$. Then, there exist $Q_{U},Q_{V}\in Prim(R)$ such that $Q_{U}\in U$ and $Q_{V}\in V$. Also we have $U=Prim(R)-V_{rad}(E)$ for $E\subseteq R$  since $U$ is open. Then we get $Q_{U}\notin V_{rad}(E)$, that is, $E\nsubseteq \sqrt{Q_{U}}.$ Hence $E\nsubseteq \sqrt{N(R)}$ because $N(R)\subseteq \sqrt{Q_{U}}\ $. This implies that $N(R)\notin V_{rad}\left( E\right),$ namely, $N(R)\in U.$ Similarly, $N(R)\in V$. Hence, we obtain $U\cap V\neq \emptyset .\
$Suppose that $N(R)\ $is not primary. Thus $N(R)$ is not prime. This means that there exist $a,b\in R\ $such
that $a,b\in R-N(R)\ $and $ab\in N(R).\ $Then we have $X_{a}\neq \emptyset \ $, $X_{b}\neq \emptyset \ $and $X_{ab}=\emptyset .\ $But, $X_{a}\cap
X_{b}=X_{ab}=\emptyset \ $for nonempty two open sets $X_{a}\ $and $X_{b}.\ $ Consequently, $Prim(R)$ is not irreducible.
\end{proof}
\definition {[7] A ring $R$ is said to be a $W$-ring if each ideal of $R$ may be uniquely represented as an intersection of finitely many primary ideals.}
\definition{A topological space $X$ is sober if every irreducible closed subset has a unique generic point.}
\theorem{Let $R$ be a $W$-ring. Then $Prim(R)$ is a $T_{0}-$space if and only if it is a sober space. }
\proof{Let $I,J$ be elements of $Prim(R).$ It is sufficient to prove that if $Prim(R)$ is a $T_{0}-$space, then it is sober. By Proposition 2(1), we have $V_{rad}(I)=Cl(\{I\})$. Since $R$ is a $W$-ring, every irreducible closed subspace of $Prim(R)$ is in the form of $V_{rad}(I)$ by Proposition 5. Suppose that $V_{rad}(I)=Cl(\{J\})$ for any $J\in Prim(R).$ Thus $I=J,$ since $Prim(R)$ is a $T_{0}-$space, that is, $V_{rad}(I)$ has a unique generic point.}
\definition{A topological space $X$ is spectral if,
\begin{enumerate}
  \item $X$ is quasi-compact,
  \item $X$ is sober,
  \item The family of quasi-compact open subsets of $X$ is closed under finite intersection and a base for the topology.
\end{enumerate}
}
\theorem{Let $R$ be a $W$-ring. Then $Prim(R)$ is a spectral space if and only if $Prim(R)$ is a $T_{0}-$space.}
\proof{Let $R$ be a $T_{0}-$space. We know that $Prim(R)$ is quasi-compact by Theorem 3 and $Prim(R)$ is a sober space by Theorem 8.  Additionally, the family $F=\{X_{r}:r\in R\}$ is quasi-compact open subsets of $Prim(R)$ by Theorem 2(4) and forming a base of a topology on $Prim(R)$. Also, this family is closed under finite intersection by Theorem 2(2).}

\begin{definition}
A topological space $X$ is supercompact if every open covering of $X$ contains $X.$
\end{definition}
We inspire from [8, Theorem 3.2] to prove following theorem.
\begin{theorem}
A ring $R$ is local if and only if $Prim(R)$ is a supercompact space.
\end{theorem}

\begin{proof}
Let $M$ be the unique maximal ideal of $R$ and $\left\{U_{i}\right\} _{i\in \Delta }\ $be an open covering of $Prim(R)$ such that $U_{i}\neq Prim(R)$ for all $i\in \Delta .\ $
As $U_{i}$ is open for all $i\in \Delta$, there exists a proper ideal $I_{i}\ $of $R\ $such that $U_{i}=Prim(R)-V_{rad}\left(I_{i}\right).$ Since $M$ is the only maximal ideal of $R$, $I_{i}\subseteq M\subseteq \sqrt{M}$. Thus, $M\notin U_{i}\ $for all $i\in \Delta \ $but $M\in Prim(R),$ that is a contradiction. Conversely, let $Prim(R)$ be a supercompact space. Assume that $\left\{
M_{i}\right\} _{i\in \Delta }\ $is a family of maximal ideals of $R\ $with $\left\vert
\Delta \right\vert >1.\ $Then, we have $\bigcap\limits_{M_{i}\in Max\left( R\right) }V_{rad}\left( M_{i}\right)=\emptyset \ $. Hence, $Prim(R)=Prim(R)-\bigcap\limits_{M_{i}\in Max\left( R\right) }V_{rad}\left( M_{i}\right)=\bigcup\limits_{M_{i}\in Max\left( R\right) }\left( Prim(R)-V_{rad}\left( M_{i}\right) \right) .\ $Therefore, there exists an $M_{j}\in Max\left( R\right) \ $such that $Prim(R)=Prim(R)-V_{rad}\left( M_{j}\right) \ $. Hence we obtain that $V_{rad}\left(M_{j}\right) =\emptyset ,\ $which is a contradiction.
\end{proof}

\end{document}